\documentclass[11pt]{amsart}
\usepackage{amsfonts}
\usepackage{amssymb}

\def\a{\alpha}		 \def\b{\beta}		  \def\g{\gamma}
		 		 \def\e{\varepsilon}
\def\la{\lambda}	 \def\om{\omega}		
	 \def\s{\sigma}			
		 	  
\def\z{\zeta}		 			

	 \def\C{{\mathbb C}}
\def\D{{\mathbb D}}	 
	 \def\N{{\mathbb N}}
	 \def\R{{\mathbb R}}
\def\T{{\mathbb T}}	 

	 \def\cb{{\mathcal B}}

\def\({\left(}		 \def\){\right)}

\newtheorem{prop}{\sc Proposition}
\newtheorem{lem}{\sc Lemma}
\newtheorem{thm}{\sc Theorem}

\newtheorem{rk}{\sc Remark}
\newtheorem{defn}{\sc Definition}

\newcommand{\bin}[2]{
 \left(
  \begin{array}{@{}c@{}}
	#1 \\ #2
  \end{array}
 \right)			 }
\begin{document}
\title[The Krzy\.{z} conjecture revisited]
{The Krzy\.{z} conjecture revisited}
\author[Mart\'{\i}n]{Mar\'{\i}a J. Mart\'{\i}n}
\address{University of Eastern Finland, Department of Physics and Mathematics, P.O. Box 111, 80101 Joensuu, Finland} \email{maria.martin@uef.fi}
\author[Sawyer]{Eric T. Sawyer}
\address{Department of Mathematics, McMaster University, Hamilton, ON  48824,
Canada} \email{sawyer@mcmaster.ca}
\author[Uriarte-Tuero]{Ignacio Uriarte-Tuero}
\address{Department of Mathematics, Michigan State University, East Lansing,
MI 48824, U.S.A.} \email{ignacio@math.msu.edu}
\author[Vukoti\'c]{Dragan Vukoti\'c}
\address{Departamento de Matem\'aticas, Universidad Aut\'onoma de
Madrid, 28049 Madrid, Spain} \email{dragan.vukotic@uam.es}
 \urladdr{http://www.uam.es/dragan.vukotic}
\date{17 December, 2014}
\begin{abstract}
The Krzy\.z conjecture concerns the largest values of the Taylor coefficients of a non-vanishing analytic function bounded by one in modulus in the unit disk. It has been open since 1968 even though  information on the structure of extremal functions is available. The purpose of this paper is to collect various conditions that the coefficients of an extremal function (and various other quantities  associated with it) should satisfy if the conjecture is true and to show that each one of these properties is equivalent to the conjecture itself. This may provide several possible starting points for future attempts at solving the problem.
\end{abstract}
\maketitle
\par
\section*{Introduction}
\par
\textbf{Formulation of the problem}. Denote by $\D$ the unit disk and by $\cb_\ast$  the class of all analytic functions $f$ in $\D$ such that $0<|f(z)|\le 1$ for all $z$ in $\D$. Consider the extremal problem of determining the following supremum:
\begin{equation}
 M_n = \sup \{|f^{(n)}(0)|/n!\,\colon\,f\in\cb_\ast\}\,, \quad n\ge 1\,.
 \label{eqn-extr-prob}
\end{equation}
A standard argument involving normal families shows that the supremum $M_n$ is attained for some function $f$. Any such function will be called an \textit{extremal function}. For every $n\ge 1$, the function
\begin{equation}
 f_n(z) = e^{(z^n-1)/(z^n+1)} = \frac1{e}+\frac2{e} z^n +\ldots
 \label{extr-fcn}
\end{equation}
shows that $M_n\ge 2/e$. In 1968 the late Polish mathematician Jan Krzy\.{z} \cite{Krz} suggested that $M_n = 2/e$ should hold for all $n\ge 1$, with equality only for the function $f_n$ given by \eqref{extr-fcn} and its rotations: $\a f(\g z)$, $|\g| = |\a| = 1$. This is known as the \textit{Krzy\.z conjecture}.
\par
Whenever $|\a|=|\g|=1$, it is plain that $f(z)= \sum_{j=0}^\infty a_j z^j$ is an extremal function if and only if $\g f(\a z)=\g \sum_{j=0}^\infty \a^j a_j z^j$ is extremal. Therefore
\begin{equation}
 M_n = \sup \{\mathrm{Re\,}a_n\,\colon\,f\in\cb_\ast\,, \ a_0>0\}\,, \quad n\ge 1\,.
 \label{eqn-extr-ref} 									
\end{equation}
Since the rotation in the argument involving $\a$ does not affect $a_0$, the coefficients of an extremal function for \eqref{eqn-extr-ref} must actually satisfy the condition $\mathrm{Re\,} a_n=|a_n|$ and hence $a_n>0$ (otherwise an appropriate rotation would yield a function in $\cb_\ast$ for which Re\,$a_n>M_n$). This observation will be used often.
\par\medskip
\textbf{Significance of the problem}. The importance of the question  stems from a number of its relationships with some fundamental results or concepts in geometric function theory. We review some of them here. For $f\in \cb_\ast$, write $f=e^g$, where Re\,$g<0$ in $\D$ and (after a suitable normalization) we may also assume that $g(0)$ is real and negative. Then the normalized function $g/g(0)$ belongs to the class $P$ of normalized functions with positive real part, hence its coefficients are bounded by $2$ in modulus by Carath\'eodory's lemma \cite[Chapter~2]{Du2}. Thus, the Krzy\.z conjecture can be thought of as an ``exponential analogue'' of the Carath\'eodory lemma.
\par
From the fact that every Taylor coefficient of $f$ can be computed recursively in terms of its previous coefficients and those of $g$, a further  relationship stems with the Faber polynomials and Grunsky's inequalities \cite[Chapter~5]{Du2}.
\par
The statement of the Krzy\.z conjecture also provides a curious subordination relation. The atomic singular inner function $S(z)= e^{(z+1)/(z-1)}$ is a universal covering map of the punctured disk $\D\setminus\{0\}$. A moments' thought reveals that every analytic function $f$ of the disk into the punctured disk is subordinated to $S$ in the sense that it can be written in the form $f=e^{(h+1)/(h-1)}$ for some analytic function $h$ from $\D$ into itself. (Let us point out here that one can also write $f=e^{(h-1)/(h+1)}$ and in this paper we will use whichever one of the two forms is more convenient for a specific purpose.) Subordination is a central topic in geometric function theory (see, for example, Chapter~6 of Duren's monograph \cite{Du2}). The basic principles in this theory state roughly that, whenever a function is subordinated to another then some of the initial Taylor coefficients will be smaller than those of the superordinate  function and some mean value of the coefficients should be smaller. The Krzy\.z conjecture states that $n$-th Taylor coefficients of any admissible $f$ cannot be larger than the corresponding coefficient of $S\circ \s_n$, where $\s_n (z)=z^n$.
\par
Some parallelism can also be observed between the Krzy\.z conjecture and the famous Bieberbach's conjecture (now de Branges' theorem) for the $n$-th coefficient of a univalent function in the class $S$; \textit{cf.\/} \cite{Du2}, \cite{dB}. Actually, in \cite{HSZ} a differential equation was devised which would govern the coefficients for the Krzy\.z problem in much the same way the celebrated Loewner's equation \cite{Du2} did for the Bieberbach's conjecture. A different approach along these lines was tried in \cite{R2}. This suggests that the relative similarity between the two problems may not be so superficial and may also nurture some hopes that, if the Krzy\.z conjecture had a ``simple'' proof then perhaps so would the de Branges theorem.
\par
The Krzy\.z  conjecture has attracted the attention of a number of mathematicians and has been mentioned in a number of relevant surveys; see Lewandowski and Szynal \cite{LS1} or B\'en\'eteau and D. Khavinson \cite{BK}, for example. Various master theses or doctoral dissertations have been, at least partially,  devoted to this problem, \textit{e.g.}, those of Ermers \cite{E} under the direction of A.C.M.~v.~Rooij and R.A.~Kortram, Romanova \cite{R1,R2} under the supervision of D.V.~Prokhorov, and Williams \cite{W} under R.~Barnard. Closely related problems had been considered earlier in Terpigoreva's thesis (\textit{cf.\/} \cite{Te}, for example) under the supervision of S.Ya.~Khavinson.
\par\smallskip
\textbf{Further motivation and brief history of the conjecture}. By the basic theory of Hardy spaces $H^p$ (see \cite{Du1} or \cite{G}), the coefficients of any function $f$ analytic in $\D$ and bounded by one in modulus must satisfy
$$
 \sum_{j=0}^\infty |a_j|^2 = \|f\|_2^2 \le \|f\|_\infty^2 = \sup_{z\in\D} |f(z)|^2 \le 1\,.
$$
Thus, if two coefficients of $f$ had the property $|a_k|$, $|a_m|\ge 2/e$, $k\neq m$, then we would have
$$
 1\ge |a_k|^2 + |a_m|^2 \ge 2 \(\frac{2}{e}\)^2 = \frac{8}{e^2} >1\,,
$$
which is absurd. Thus, no more than one coefficient of $f$ can possibly be bigger than $2/e$. This seems to speak in favor of the conjecture as one would naturally expect stronger statements to hold for functions in $\cb_\ast$.
\par
A well-known estimate shows that $|a_n|\le 1-|a_0|^2$ for all $n\ge 1$ whenever $f$ is analytic and bounded by one in $\D$; see \cite{AW} or \cite{W}. We wish to stress that for the class $\cb_\ast$  uniform bounds on $M_n$ strictly smaller than one are known. Horowitz \cite{Ho} showed that
$$
 M_n\le 1 - \frac1{3\pi}+ \frac4{\pi} \sin \frac1{12} = 0.99987\ldots \,, \quad n\in\N\,.
$$
This was later improved slightly to 0.9991\ldots by Ermers \cite{E}. Both values are obviously still far from the desired bound $2/e = 0.73575888 \ldots$ Some useful asymptotic bounds were obtained in \cite{PR}.
\par
As for the exact bounds, it was known as early as in 1934 that $M_1=2/e$. The first known record of this seems to appear in \cite{LFR}; see also \cite{SR} or \cite{AW}. However, so
far the Krzy\.z conjecture has only been proved for $n\le 5$. Krzy\.z \cite{Krz} and Reade, as well as MacGregor (unpublished), were the first to prove the statement for $n=2$. Hummel, Scheinberg, and Zalcman \cite{HSZ} later gave a new proof in this case and also solved the problem for  $n=3$. Tan \cite{Ta} and Brown \cite{Br2} proved the conjecture for $n=4$ and Samaris \cite{Sm} did it for $n=5$. As far as we know, the conjecture remains open for all other values of $n$.
\par
Some partial progress on the problem was obtained by a number of other authors. We mention the papers by Brown \cite{Br1} and Peretz \cite{Pe1, Pe2} who provided the proof under some additional assumptions on the coefficients of extremal functions. We also mention \cite{KS}, \cite{LS2}, \cite{PS}, and \cite{Sz}. Krushkal's unpublished preprint \cite{Kr}, in spite of some gaps found in it, contains a wealth of geometric and analytic ideas which may be useful for a further study of problems of this type.
\par\medskip
\textbf{Some known qualitative results}. In their influential paper \cite{HSZ}, Hummel, Scheinberg, and Zalcman obtained various relationships between the coefficients of an extremal function. They showed that the extremal functions are of the form
\begin{equation}
 f(z)= e^{\sum_{j=1}^N r_j \frac{\a_j z + 1}{\a_j z - 1}}\,, \quad 1\le N \le n\,, \quad r_j> 0\,, \quad |\a_j|=1, \ 1\le j\le N\,.
 \label{str-extr-fcn}
\end{equation}
The above structure is not too surprising in view of the general knowledge of extremal problems developed by S.Ya.~Khavinson and first published in Russian in the early 60's; see \cite{KhS} or \cite{BK}. Namely, as a limit case of the result from \cite{KhS} for functions in the unit ball of $H^p$ spaces as $p\to\infty$, it is possible to deduce \eqref{str-extr-fcn}. The authors of \cite{HSZ} gave their own proof of \eqref{str-extr-fcn} and mentioned two other possible proofs. Kortram \cite{Ko} later gave yet another proof. We also mention a result on the structure of extremal functions for more general coefficient problems that follows directly from an earlier work of Hummel \cite{Hu}, as was pointed out by Sakaguchi \cite[Theorem~C]{Sk}.
\par
Unfortunately, in spite of such a clear structure of extremal functions there are simply too many parameters to control here, so it is not at all immediate that the desired symmetry holds in \eqref{str-extr-fcn}:
\[
 N=n\,; \quad r_j=\frac1{n}\quad \mathrm{and} \quad \a_j = \a\,e^{2\pi j i/n}\,, \quad j=1,2,\ldots,n\,, \quad |\a|=1\,.
\]
Hence the problem remains open in spite of all the information available. However, it seems that most experts believe that the conjecture is true.
\par\medskip
\textbf{Contents and organization of the paper}. Our aim is to consider the problem from a different viewpoint. First, we will show that the Taylor coefficients of an extremal function for \eqref{eqn-extr-ref}, as well as the zeros of some polynomials associated with it in a natural way, must satisfy certain inequalities. Next, it turns out that whenever equality holds in any one of these inequalities, the conjecture is true. This is the main purpose of this paper. Therefore our main statement, Theorem~\ref{thm-equiv}, can be viewed as a reformulation of the conjecture in many ways.
\par
Our proofs are typically based on variational methods, similar to those employed in \cite{HSZ}, on the Riesz factorization for Hardy spaces, on the Fej\'er lemma on polynomials with positive real part on the unit circle, and on the Carath\'eodory lemma for analytic functions with positive real part in $\D$.
\par
The paper is organized as follows. We first review some known facts, several of them with new proofs, and collect other useful information in the section on preliminary results. In the final section we formulate explicitly the main result and give its detailed proof, together with various comments.
\par\medskip
\textsc{Acknowledgments}. {\small The authors would like to thank the referee for the most careful reading of the manuscript, some helpful suggestions, and for pointing out a mistake in the first draft of the paper. Thanks are also due to Alexandru Aleman, Roger Barnard, Catherine B\'en\'eteau, Mario Bonk, Peter L. Duren, John B. Garnett, Daniel Girela, H\aa kan Hedenmalm, Dmitry Khavinson, Yuri\u{\i} Lyubarski\u{\i}, Donald Marshall, Steffen Rohde, Peter Sarnak, Kristian Seip, Alexander Vasil'ev, and Lawrence Zalcman, either for their interest in the problem and encouragement or for some stimulating conversations on the subject or for useful information on the references.
\par
The first, third and fourth authors thankfully acknowledge partial
support from MINECO grant MTM2012-37436-C02-02, Spain. Starting in 2014 the first author was supported by Academy of Finland grant 268009. The second author was supported by NSERC, Canada. The third author was also supported by NSF grant DMS-0901524, by NSF CAREER, Award No. DMS-1056965, and Sloan Research Fellowship, USA. The fourth author was partially supported by the European ESF Network HCAA (``Harmonic and Complex Analysis and Its Applications'') during the period November 2009 - April 2012.
\par
Parts of this work were done at various research institutes:
\par
- in September-October of 2011 while the first, third and fourth authors participated in the special program ``Complex Analysis and Integrable Systems'' held at Institut Mittag-Leffler in Djursholm, Sweden,
\par
- in August of 2012 during a workshop held at the same institution and attended by the first and the fourth authors,
\par
- in March of 2012 while the fourth author attended the workshop ``Operator Related Function Theory'' at the Erwin Schr\"odinger Institute in Vienna, Austria,
\par
-  in the Spring of 2013 during the third author's stay at IPAM in Los Angeles for the special program ``Interactions Between Analysis and Geometry''.
\par
Part of the work was also done during various short visits by the third author to both Universidad Aut\'onoma de Madrid and MacMaster University and of the second author to Michigan State University during the period 2012-14. The authors would like to thank all these institutions for their hospitality and partial financial and technical support and for providing a stimulating environment for doing mathematical research.}
\par
\section{Various preliminary observations and results}
\par
\textbf{The simplest case}. We formulate the answer in the easiest case $n=1$ as a lemma and also present a very simple proof which seems different from the ones  published before, for example, from \cite{SR}. This is done not only for the sake of completeness but also because it will be needed later.
\par
\begin{lem}
 \label{lem-n=1}
If $f\in\cb_\ast$ then $|a_1|\le \frac2{e}$. Equality holds only for the
functions of the form
\begin{equation}
 f(z)=\g e^\frac{\a z+1}{\a z-1}\,, \qquad |\a|=|\g|=1\,.
\end{equation}
\end{lem}
\par
\par\smallskip
\begin{rk} \label{rk-n=1}
Note that, under the normalization \eqref{eqn-extr-ref}, it is easy to check that equality holds in Lemma~\ref{lem-n=1} if and only if $f(z)=e^{\frac{z-1}{z+1}}$ (when $-\a=1=\g$).
\end{rk}
\begin{proof}
If $f\in \cb_\ast$ then we can write $f=e^g$ where $g$ is a function
analytic in $\D$ and with negative real part, hence $g=(h+1)/(h-1)$,
for some $h$ with $\|h\|_\infty\le 1$. A direct computation shows that
\[
 f^\prime = - \frac{2 h^\prime}{(h-1)^2} e^{\dfrac{h+1}{h-1}}\,,
\]
hence by the Schwarz-Pick lemma
\[
 |f^\prime(0)| = \frac{2 |h^\prime (0)|}{|1-h(0)|^2}\,
 e^{\mathrm{Re\,} \frac{h(0)+1}{h(0)-1}} \le 2 \frac{1-|h(0)|^2}{|1-h(0)|^2}\,
 e^{-\dfrac{1-|h(0)|^2}{|1-h(0)|^2}} \le \frac2{e}
\]
since elementary calculus shows that the function $u(x)=2x\,e^{-x}$
considered in $[0,+\infty)$ attains its maximum at the point $x=1$.
\par
The case of equality requires some analysis. In order for equality to hold, we must have equality in the Schwarz-Pick lemma so $h$ has to be a disk automorphism. Also, $h(0)$ must belong to the set of all $w$ such that $1-|w|^2=|1-w|^2$, which is the horocycle $\{z\,\colon\, |z-\frac12|= \frac12\}$. Now note that the linear fractional (M\"obius) map $w\mapsto (w+1)/(w-1)$ maps this horocycle onto the vertical line Re\,$z=-1$. In particular, we have
\[
 \frac{h(0)+1}{h(0)-1}=-1+i c\,, \quad c\in\R\,.
\]
But the linear fractional map $(h+1)/(h-1) - i c$ maps the unit disk conformally onto the left half-plane and the origin to the point $-1$. It follows that
\[
 g(z)=\frac{h(z)+1}{h(z)-1}=i c +\frac{\a z+1}{\a z-1}\,, \quad |\a|= 1\,.
\]
The desired structure of $f$ is observed immediately.
\end{proof}
\textbf{Some useful recurrence relations}. We continue with the following simple computation which is, for example, used to derive Grunsky's inequalities \cite[p.~143]{Du2}.
\par
\begin{lem}
 \label{lem-exp}
If $f$ and $g$ are analytic in $\D$ and
$$
 f=e^g\,, \quad	 f(z)= \sum_{j=0}^\infty a_j z^j\,, \quad g(z)= \sum_{j=0}^\infty b_j z^j\,,
$$
then $a_0=e^{b_0}$ and
\begin{equation}
 a_n = \sum_{k=0}^{n-1} \frac{n-k}{n}\,a_k b_{n-k} = \sum_{j=1}^{n} \frac{j}{n} a_{n-j} b_j\,, \ n\ge 1\,.
 \label{rec-exp}
\end{equation}
\end{lem}
\begin{proof}
Differentiation of $f=e^g$ yields $f^\prime = f g^\prime$. Upon
differentiating $n-1$ more times and applying the Leibniz formula we
get
$$
 f^{(n)} = \sum_{k=0}^{n-1} \bin{n-1}{k} f^{(k)} g^{(n-k)} \,.
$$
After evaluating both sides at the origin and dividing both sides by
$n!$, we get the desired formula.
\end{proof}
From the above lemma it easily follows that, for example,
\begin{eqnarray*}
 a_1 &=& a_0 b_1\,, \quad a_2 = a_0 \(\frac{b_1^2}{2} + b_2\)\,, \quad a_3 = a_0 \(\frac{b_1^3}{6} + b_1 b_2 + b_3\) \,,
\\
 a_4 &=& a_0 \(\frac{b_1^4}{24} + \frac{b_1^2 b_2}{2} + b_1 b_3 + \frac{b_2^2}{2} + b_4\) \,, \ldots
\end{eqnarray*}
This is easily generalized to obtain the following structural formula which essentially reduces to the well-known Fa\`a di Bruno formula \cite{J} on differentiation of composite functions (for functions of exponential type). We could even be more specific about the values of some of the coefficients but this will not be needed in the paper.
\par\smallskip
\begin{prop}
 \label{prop-polyn-a}
Let $f$ and $g$ and their coefficients be as in Lemma~\ref{lem-exp}. Then
$$
 a_n = a_0\,P_n (b_1,b_2,\ldots,b_n)\,,
$$
where, for each $n\ge 1$, $P_n$ is a polynomial of the form
$$
 P_n (b_1,b_2,\ldots,b_n) = \displaystyle \sum_{\substack{1\le m\le n, \\ \sum_{j=1}^m i_j(n) = n}} c_{i_1(n),i_2(n),\ldots,i_m(n)} b_{i_1(n)} b_{i_2(n)}\ldots b_{i_m(n)} \,,
$$
where all $i_j(n) \in\N$ and the coefficients $c_{i_1(n),i_2(n),\ldots,i_m(n)}$ are all strictly positive. In particular, whenever $1\le m\le n$ and $\sum_{j=1}^m i_j(n) = n$, the term containing the product $b_{i_1(n)} b_{i_2(n)}\ldots b_{i_m(n)}$ actually appears in the expression for $P_n$ with a non-zero coefficient in front of it.
\end{prop}
\begin{proof}
The statement follows easily by induction. It is obviously true for $n=1$ as $P_1(b_1)=b_1$. Now let $n\ge 2$ and suppose that the claim is true for every $k$ with $1\le k<n$; that is:
$$
 P_k (b_1,b_2,\ldots,b_k) = \displaystyle \sum_{\substack{1\le m_k \le k, \\ \sum_{j=1}^{m_k} i_j(k) = k}} c_{i_1(k),i_2(k),\ldots,i_{m_k}(k)} b_{i_1(k)} b_{i_2(k)}\ldots b_{i_{m_k(k)}} \,,
$$
whenever $1\le k<n$. By \eqref{rec-exp} we have
\begin{eqnarray*}
 \frac{a_n}{a_0} &=& \sum_{k=0}^{n-1} \frac{n-k}{n} \frac{a_k}{a_0} b_{n-k} = b_n + \sum_{k=1}^{n-1} \frac{n-k}{n} P_k (b_1,b_2,\ldots,b_k)\,b_{n-k}
\\
 &=& b_n + \sum_{k=1}^{n-1} \frac{n-k}{n} \displaystyle \sum_{\substack{1\le m_k \le k, \\ \sum_{j=1}^{m_k} i_j(k) = k}}
 c_{i_1(k),i_2(k),\ldots,i_{m_k}(k)} b_{i_1(k)} b_{i_2(k)}\ldots b_{i_{m_k(k)}}  b_{n-k}
\\
 &=& b_n + \displaystyle \sum_{\substack{2\le m\le n, \\ \sum_{j=1}^m i_j(n) = n}} c_{i_1(n),i_2(n),\ldots,i_m(n)} \, b_{i_1(n)} b_{i_2(n)}\ldots b_{i_m(n)} \,,
\end{eqnarray*}
which completes the proof of the inductive step. The last identity in the above string of equalities follows from the obvious fact that a positive integer $n$ can be written as a sum of at least two positive integers:
$$
 n= i_1(n)+i_2(n)+\ldots+i_m(n)\,, \quad m\ge 2\,,
$$
if and only if one of the numbers $i_j(n)=n-k$ with $0\le k<n$ and the sum of the remaining ones is $k$. Note also that, when collecting terms with the same product $b_{i_1(n)} b_{i_2(n)}\ldots b_{i_m(n)}$ (which may appear in several summands in the last line in the display above) no cancelation can occur because all the coefficients $\frac{n-k}{n}\,c_{i_1(k),i_2(k),\ldots,i_{m_k}(k)}$ are positive.
\end{proof}
\par\smallskip
For $n\ge 1$ and considering $a_0$ as a constant, \eqref{rec-exp} allows us also to express the coefficients $b_n$ as a polynomial of $a_1$,\ldots,$a_n$; for example:
\begin{eqnarray*}
 b_1 &=& \frac{a_1}{a_0}\,, \qquad b_2 = \frac1{a_0} \(a_2 - \frac12 a_1 b_1 \) = \frac1{a_0} a_2 - \frac1{2a_0^2} a_1^2\,,
\\
 b_3 &=& \frac1{a_0} \(a_3 - \frac13 a_2 b_1 - \frac23 a_1 b_2\) =
 \frac1{a_0} a_3 - \frac1{a_0^2} a_1 a_2 + \frac1{3 a_0^3} a_1^3 \,, \ldots
\end{eqnarray*}
The difference with respect to the previous lemma is that some coefficients are no longer positive. However, all possible terms $a_{i_1} a_{i_2}\ldots a_{i_m}$ with $i_1+i_2+\ldots+i_m=n$ are present in the formula for each $b_n$ and the signs of the coefficients are easy to control: they are positive in front of a product of an odd number of terms and negative in front of a product of an even number of terms $a_i$. It turns out that in computing $b_n$ when we sum up similar terms, coming from different summands but containing the same product, no cancelation of the coefficients in front of two similar terms occurs because these coefficients will have the same sign. This is easily checked when computing $b_3$ and $b_4$ and can be proved formally without difficulty. We formulate the precise statement as follows.
\begin{prop}
 \label{prop-polyn-b}
Let $f$ and $g$ and their coefficients be as in Lemma~\ref{lem-exp} and let $a_0>0$ be fixed. Then
$$
 b_n = Q_n (a_1,a_2,\ldots,a_n)\,,
$$
where the polynomial $Q_n$ has the form
\begin{eqnarray*}
 && Q_n (a_1,a_2,\ldots,a_n) =
\\
 && \displaystyle \sum_{\substack{1\le m\le n, \\ \sum_{j=1}^m i_j(n)} = n} (-1)^{m+1} c_{i_1(n),i_2(n),\ldots,i_m(n)}(a_0)\,a_{i_1(n)} a_{i_2(n)}\ldots a_{i_m(n)}\,,
\end{eqnarray*}
where every $c_{i_1(n),i_2(n),\ldots,i_m(n)} (a_0)$ is positive. In particular, whenever $n=\sum_{j=1}^m i_j(n)$, the term containing the product $a_{i_1(n)} a_{i_2(n)}\ldots a_{i_m(n)}$ effectively appears in the expression for $Q_n$ above (with a non-zero coefficient in front of it).
\end{prop}
\begin{proof}
A proof can again be given by induction, similar to Proposition~\ref{prop-polyn-a} but here we should also explain the sign changes $(-1)^{m+1}$.
\par
A simple inspection of the formulas preceding this result shows that the statement is obviously true for $n=1$ and $n=2$. Suppose $n>1$ and the claim is true for every $j$ with $1\le j<n$. Again by \eqref{rec-exp}, we have
\begin{eqnarray*}
 b_n &=& \frac{1}{a_0} \( a_n - \sum_{j=1}^{n-1} \frac{j}{n} a_{n-j} b_j \)
\\
 &=& \frac{1}{a_0} \( a_n - \sum_{j=1}^{n-1} \frac{j}{n} Q_{j} (a_1,a_2,\ldots,a_j)\,a_{n-j} \)
\\
\end{eqnarray*}
and everything will be similar to the previous result with one difference: whenever a new product $a_{i_1(j)} a_{i_2(j)}\ldots a_{i_m(j)} a_{n-j}$ is created from the earlier term $a_{i_1(j)} a_{i_2(j)}\ldots a_{i_m(j)}$ (that is, whenever there is an extra factor), a sign change occurs simultaneously. This explains the appearance of the factor $(-1)^{m+1}$ in front of all terms containing similar products.
\end{proof}
\par\smallskip
\par\medskip
\textbf{The structure of extremal functions}.
We have already observed that all functions in the class $\cb_\ast$ are of the form $f=e^{(h+1)/(h-1)}$ for some analytic self-map $h$ of $\D$. Now recall Carath\'eodory's theorem (see \cite[Theorem~2.1]{G} or \cite[Theorem~IV.24]{Ts}) which says that, for any given $n\ge 1$ and an analytic function $h$ whose modulus is bounded by one in the disk, with $h(z)= \sum_{k=0}^\infty a_k z^k$, one can find a finite Blaschke product $B$ of degree at most $n+1$ with
$$
 B(z)=\sum_{k=0}^n a_k z^k + \sum_{k=n+1}^\infty c_k z^k\,.
$$
By combining Lemma~\ref{lem-exp} and this result, it is actually not difficult to see that \textit{there exists\/} an extremal function which is of the form
\begin{equation}
 f = e^{(B+1)/(B-1)}
 \label{extr-fcn-bl-form}
\end{equation}
where $B$ is a Blaschke product of degree at most $n+1$.
\par
It takes a further step to deduce that actually \textit{every\/} extremal function is as in \eqref{extr-fcn-bl-form} but with the degree of $B$ at most $n$ to obtain the theorem on the structure of extremal functions as given in \cite{HSZ}. We have already mentioned several references on this starting with S.Ya. Khavinson's work. Let us also mention that it is possible to combine the open mapping theorem for non-constant analytic functions and the Toeplitz-Carath\'eodory theorem on the coefficients of analytic functions from the disk into the right half-plane \cite[Theorem~IV.22]{Ts} to give yet another proof of this statement, as was shown to us by Donald Marshall. Here we only recall again the exact statement of this result from \cite{HSZ}:
\smallskip\par\noindent
{\sc Theorem}. \textit{Every extremal function for the Krzy\.z problem \eqref{eqn-extr-ref} is of the form
\[
 f(z)= e^{\sum_{j=1}^N r_j \frac{\a_j z + 1}{\a_j z - 1}}\,,
\]
for certain values of the parameters considered which satisfy}
\[
 1\le N \le n\,; \quad r_j> 0 \quad and \quad |\a_j|=1, \ 1\le j\le N\,.
\]
(Without loss of generality, we may assume that $\a_j\neq \a_k$ whenever $j\neq k$.)
\par\smallskip
It should be noted that every function as above is actually of the form \eqref{extr-fcn-bl-form}. Indeed, note that all the fractions $\frac{\a_j z + 1}{\a_j z - 1}$ map the unit disk to the left half-plane and the unit circle to the imaginary axis and since all $r_j>0$ the same is true of the exponent
$$
 g(z) = \sum_{j=1}^N r_j \frac{\a_j z + 1}{\a_j z - 1}\,.
$$
Invoking again the conformal map of the disk onto the left half-plane, we see that $B=(g+1)/(g-1)$ maps the unit disk to itself, the unit circle into itself and has $N$ zeros in the disk counting the multiplicities (since $g$ takes on the value $-1$ exactly $N$ times, which is easily seen by inspecting the resulting polynomial equation).  It follows that $B$ is a finite Blaschke product of degree $N$ by the well-known characterization of such functions \cite[p.~6]{G}. Solving for $g$, we see that all extremal functions are actually of the form
\eqref{extr-fcn-bl-form} where $B$ is a finite Blaschke product of degree $N\le n$.
\par\medskip
\textbf{On the Taylor coefficients of an extremal function}. We now recall some important facts. Parts (a) and (b) of the statement below may not have been recorded explicitly in the literature while (c) and (d) were deduced on p.~173 of \cite{HSZ}. For the sake of completeness, we include a simple proof of both facts by an elementary variational method, \textit{i.e.}, using differentiation with respect to a parameter.
\par\smallskip
\begin{prop}
 \label{prop-var}
Let $n>1$ and let $f(z)=\sum_{j=0}^\infty a_j z^j$ be an extremal function for \eqref{eqn-extr-ref}. (Recall that then, as observed earlier, $a_n$ is real and $a_n>0$.) Then
\begin{itemize}
\item[(a)]
If $u$ is an arbitrary analytic function with negative real part in $\D$ and $u(z)=\sum_{n=0}^\infty c_n z^n$, then
$$
 \mathrm{Re\,} \{a_n c_0 + a_{n-1} c_1 + \ldots + a_0 c_n\} \le 0\,.
$$
\item[(b)]
When $f=e^g$ and the coefficients of $g$ are denoted by $b_j$, we have
\begin{equation}
 \mathrm{Re\,} \{a_n b_0 + a_{n-1} b_1 + \ldots + a_0 b_n\} = 0\,.
 \label{cond-coeff}
\end{equation}
\item[(c)]
$a_n\ge 2 a_0$.
\item[(d)]
The polynomial $P(z) = a_n + 2 a_{n-1} z + \ldots 2 a_1 z^{n-1} + 2 a_0 z^n$ satisfies Re\,$P(z)\ge 0$ whenever $|z|\le 1$.
\item[(e)]
Moreover, if the extremal function (which is a singular inner function with finitely many atoms) has $N$ point masses: at $\a_1$, $\a_2$,\ldots,$\a_N$, $1\le N\le n$, then actually Re\,$P(\a_k)= 0$ for each $k$ with $1\le k\le N$.
\end{itemize}
\end{prop}
\begin{proof}
(a) Let $u$ be analytic in $\D$ with $\mathrm{Re\,} u < 0$ and
let $\e > 0$. Then the function $f e^{\e u} \in \cb_\ast$ and
is, hence, in contention with $f$. But
$$
 f e^{\e u} = f (1+\e u+O(\e^2)) = f + \e f u + O(\e^2) \,.
$$
By comparing the $n$-th coefficients, canceling, dividing out by $\e$ and letting $\e \to 0$, we see that the real part of the $n$-th coefficient of $f u$ is $\le 0$, which proves (a).
\par
(b) Note that for the specific choice $u=g$ above, where $f=e^g$, we are allowed to consider $\e<0$ with small absolute value and can hence obtain equality.
\par
(c) Now pick
$$
 u(z) = \frac{z^n - 1}{z^n +1} = -1+ 2 z^n - 2 z^{2n} + \ldots
$$
to deduce that
$$
 2a_0 - a_n = \mathrm{Re\,} \{2a_0 - a_n\} \le 0\,.
$$
\par
(d) Let $\la$ be an arbitrary complex number with $|\la|\le 1$. Choose
$$
 u_\la (z) = \frac{\la z +1}{\la z - 1} = - (1 + 2\la z + 2 \la^2 z^2 + 2 \la^3 z^3 + \ldots)
$$
to infer that
$$
 \mathrm{Re\,} \{a_n+2 a_{n-1}\la + \ldots +2 a_1 \la^{n-1} + 2a_0 \la^n\} \ge 0
$$
whenever $|\la|\le 1$.
\par
(e) Follows from our formula \eqref{str-extr-fcn} by another variation. Namely, for any $k$ with $1\le k\le N$, we can consider the function
\begin{eqnarray*}
 g_\e (z) &=& f(z)\,e^{\e \frac{\a_k z+1}{\a_k z-1}} = f(z)\,\(1+ \e \frac{\a_k z+1}{\a_k z-1} + O(\e^2) \)
\\
 &=& f(z)\, \(1-\e (1+2\a_k z+2 \a_k^2 z^2+\ldots + O(\e^2))\)\,, \quad \e\to 0\,,
\end{eqnarray*}
which is in competition with $f$ for being an extremal function for any small $\e$, positive or negative. From here one immediately realizes  that the $n$-th coefficient of $g_\e$ is precisely
\[
 a_n - \e P(\a_k) + O(\e^2)\,, \quad \e\to 0\,,
\]
and the real part of this function (on some open interval around $\e=0$) attains its maximum $a_n$ at $\e=0$. The statement follows easily from here.
\end{proof}
\par\smallskip
Part (c), as was observed in \cite{Pe1}, has the following corollary: if $f$ is an extremal function for \eqref{eqn-extr-ref} then its constant term enjoys the estimate $a_0\le \sqrt{2}-1 \approx 0.41421356237\ldots$. This is immediate from the inequalities mentioned earlier: $2a_0\le a_n \le 1-a_0^2$. Of course, our ultimate goal would be to show that actually $a_0=1/e \approx 0.36787944117\ldots$	
\par
If the Krzy\.z conjecture is true then the suspected extremal functions should satisfy the equality $\mathrm{Re\,}a_n=2 a_0$ since $a_0=1/e$ and $a_n=2/e$. We will now show that the converse is also true. That is, proving this fact for any extremal function is equivalent to proving the Krzy\.z conjecture. We will show that there are also many other statements equivalent to it.
\par
It is worth mentioning that several existing partial results on the Taylor coefficients of an extremal function either go in another direction or seem to use stronger initial hypotheses. For example, it was shown in \cite{Pe1} that if $n$ is odd, $f$ is extremal for \eqref{eqn-extr-prob}, and $a_1 = a_3 = \cdots = a_{n-2} = 0$ then $|a_0|\le 1/e$ and equality holds if and only if $|a_n|= 2/e$. Before proceeding on to improving this result, we need to review some basic facts.
\par\medskip
\textbf{On the coefficients of the polynomial associated with an extremal function}. Denote by $\la_k$, $1\le k\le n$, the zeros of the polynomial $P$ defined in Proposition~\ref{prop-var}. Since
$$
 P(z) = 2 a_0 \prod_{k=1}^n (z-\la_k)\,,
$$
it follows that
\begin{equation}
 a_n = P(0) = 2 (-1)^n a_0 \prod_{k=1}^n \la_k \,.
 \label{a_n-a_0}
\end{equation}
In view of Proposition~\ref{prop-var} and the fact that $a_n$,
$a_0>0$, we get
\begin{equation}
 (-1)^n \prod_{k=1}^n \la_k \ge 1 \,.
 \label{prod-roots}
\end{equation}
We actually know more: $|\la_k|\ge 1$ for all $k\in\{1,2,\ldots,n\}$.
The reason is that if for some $j$ we had $|\la_j|<1$ and $P(\la_j)=0$  then the open mapping theorem for analytic functions would imply that in any small neighborhood of $\la_j$ there is a point $z$ at which Re\,$P(z)<0$, which would contradict the fact that Re\,$P(z)\ge 0$ in
$\overline{\D}$.
\par\medskip
\textbf{The Fej\'er lemma}. Given a complex polynomial of degree $n$: $P(z)=\sum_{k=0}^n c_k z^k$, if we look at its restriction to the unit circle and write each $z$ of modulus one as $z=e^{i t}$, $t\in [0,2\pi]$, it is easy to see that Re\,$P$ is a trigonometric polynomial of degree $n$:
\begin{equation}
 T(t)=\a_0+\sum_{k=0}^n (\a_k \cos k t + \b_k \sin k t)	
 \label{trig-pol}
\end{equation}
and can, thus, have at most $2n$ zeros in $[0,2\pi]$. In particular,
from here we see the following:
\par
\textit{If the real part of a complex polynomial $P$ vanishes on the unit circle then $P$ is identically equal to a purely imaginary constant.}
\par
The following classical lemma due to Fej\'er (see \cite[p.~154--155]{Ts}) characterizes an important class of
trigonometric polynomials.
\smallskip\par\noindent
{\sc Fej\'er's Lemma}. \textit{If $T$ is a trigonometric polynomial as in \eqref{trig-pol} and $T(t)\ge 0$ for all $t\in [0,2\pi]$ then there are complex coefficients $\g_j$, $0\le j\le n$, such that}
$$
 T(t)=|\g_0+\g_1 e^{i t} +\ldots \g_n e^{i n t}|^2\,, \quad for\ all\  t\in [0,2\pi]\,.
$$
\par\medskip
\textbf{Wiener's trick}. The following argument is well-known and appears in different contexts in complex analysis. It is actually the basis of F.~W.~Wien\-er's proof of the inequality $|a_n|\le 1-|a_0|^2$ mentioned earlier; see \cite[p.~4]{Bo}, for example. Thus, we shall refer to it as the \emph{Wiener trick\/}.
\par
Given a function $f\in\cb_\ast$ with $f(z)=\sum_{k=0}^\infty a_k z^k$ and a fixed integer $n>1$, consider the primitive $n$-th root of unity: $\om=e^{2\pi i/n}$. It is routine to check that the \textit{Wiener transform\/} of $f$, given by
\begin{equation}
 W_n f(z)= \frac1{n} \sum_{k=0}^{n-1} f(\om^k z) = \sum_{k=0}^\infty a_{nk}  z^{nk}\,,
 \label{fcn-wt}
\end{equation}
is of the form $H(z^n)$, where
\begin{equation}
 H(z)=\sum_{k=0}^\infty a_{nk} z^k\,.
 \label{fcn-h}
\end{equation}
Obviously, both $W_n f$ and $H$ are analytic in $\D$ and bounded by one there. Moreover, $W_n f(0)= H(0)=a_0$ and $H^\prime (0)=a_n$. This will sometimes allow us to translate the problem for the $n$-th Taylor coefficient to the problem for the first coefficient, already solved by Lemma~\ref{lem-n=1}.
\par
Note that when $f$ is extremal, the function $H$ defined in \eqref{fcn-h} and associated with its Wiener transform $W_n f$ may or may not vanish in $\D$. (If it does not, the conjecture will easily follow as we will see later.) In either case, the standard Riesz factorization for Hardy spaces \cite[Chapter~2]{Du1} tells us that there exist a Blaschke product $B$ (possibly a constant of modulus one) and a function $G$ which is analytic and non-vanishing in $\D$ such that
\begin{equation}
 H = B G\,, \qquad \|G\|_\infty = \|H\|_\infty \le 1\,.
 \label{factor}
\end{equation}
Since for any $\a$ with $|\a|=1$ we have $H = (\a B) (\overline{\a} G)$ and $H(0)=a_0>0$, we can replace $B$ by $\a B$ and so without loss of generality we may assume that $B(0)$ is real and positive. Thus,
$$
 W_n f(z) = a_0 + a_n z^n + \ldots = B(z^n) G(z^n) = (B_0 + B_n z^n + \ldots) (C_0 + C_n z^n + \ldots)
$$
with $B_0=B(0)>0$. Obviously,
\begin{equation}
 a_0=B_0 C_0\,, \qquad a_n = B_0 C_n + B_n C_0\,,
 \label{coeff-reln}
\end{equation}
and since $a_0>0$ we see that actually $C_0>0$ as well. This discussion includes the case when $W_n f$ does not vanish in the disk, meaning that $B\equiv B_0=1$ in that case. We have already proved that $2\le a_n/a_0$ for any extremal function. We will now prove a related upper bound.
\begin{prop} \label{prop-wien}
Let $f$ be an extremal function for the Krzy\.z problem. Then, with $B$ as in \eqref{factor} and $B_0=B(0)$ normalized so that $0<B_0\le 1$, we have $$
 \frac{a_n}{a_0} \le 1 + \frac1{B_0}\,.
$$
\end{prop}
\begin{proof}
Assume the contrary:
$$
 \frac{a_n}{a_0} > 1 + \frac1{B_0}\,.
$$
In view of \eqref{coeff-reln}, this means that $a_n> a_0+C_0 = C_0(1+B_0)$. This yields
$$
 a_n (1-B_0) > C_0 (1-B_0^2) \ge C_0 |B_n| \ge C_0\,\mathrm{Re\,}B_n
$$
by the well-known inequality $|B_n|\le 1-B_0^2$ that follows from Wiener's trick. From here we get that
$$
 a_n - C_0\,\mathrm{Re\,}B_n > a_n B_0\,.
$$
Thus, again by \eqref{coeff-reln},
$$
 B_0\,\mathrm{Re\,}C_n = \mathrm{Re\,}\{a_n - C_0 B_n\} = a_n - C_0\,\mathrm{Re\,}B_n > a_n B_0\,.
$$
Since $B_0>0$, this shows that $\mathrm{Re\,}C_n > a_n$. But $G(z^n)$ belongs to the class $\cb_\ast$ and is thus in contention with $f$. This contradicts the assumption that $f$ is extremal in $\cb_\ast$.
\end{proof}
\par\medskip
\textbf{Inequalities of Carath\'eodory and Livingston type}. Denote by $P$ the class of all analytic functions $u$ in $\D$ such that Re\,$u(z)>0$ in $\D$ and $u(0)=1$. If we write the Taylor series expansion of such $u$ in the disk as
$$
 u(z) = 1 + b_1 z + b_2 z^2 + b_3 z^3 + \ldots\,,
$$
the well-known Carath\'eodory's lemma \cite[Chapter~2]{Du2} states that $|b_n|\le 2$ for all $n\ge 1$. There are many other inequalities for the coefficients in this class, several of them due to Livingston. Here we only need one such inequality which can also be deduced without much effort from \cite[Lemma~1]{L} but we give our own proof.
\begin{lem}
 \label{lem-livingst}
If $-u\in P$, $u(z) = -1 + b_1 z + b_2 z^2 + b_3 z^3 + \ldots$,
and $k\in\N$ then
$$
 \left| b_{2k} + \frac{b_k^2}{2}\right| \le 2\,.
$$
\end{lem}
\begin{proof}
It suffices to prove the inequality in the case $k=1$:
$$
 \left| b_{2} + \frac{b_1^2}{2}\right| \le 2\,.
$$
To this end, note that every $u$ such that $-u\in P$ can be written as $u=(h-1)/(h+1)$ where $h$ is an analytic function from $\D$ into itself and $h(0)=0$. Let $c_2=h^{\prime\prime}(0)/2$ be the second Taylor coefficient of $h$ at the origin; then by the inequality mentioned earlier for all analytic self-maps of the disk we have $|c_2|\le 1-|h(0)|^2=1$. By differentiating the equality
$$
 uh +u = h-1
$$
twice and taking into account that $u(0)=-1$ and $c_0=h(0)=0$, a direct computation yields
$$
2 c_2 =  b_{2} + \frac{b_1^2}{2}
$$
and the desired inequality follows.
\par
The case $k=1$ of Livingston's inequality already proved applied to the  Wiener's transform $W_k u$ yields the statement in the general case $k>1$.
\end{proof}
\par\medskip
\textbf{Inductive sets}. In some papers the Krzy\.z conjecture was proved under the additional hypotheses on an extremal function that $a_i=0$ for all $i$ belonging to some $I\subset \{1,2,\ldots,n-1\}$. Typically, ``about a half of these initial coefficients'' are assumed to vanish.
\par
More specifically, in Brown's paper \cite{Br1} on a similar but more general problem for $H^p$ spaces (see Corollary~2 and the comment that follows it in \cite{Br1}) it was shown that the assumption that $a_i=0$ whenever $1\le i< (n+1)/2$ implies the conjecture.
\par
Also, Peretz \cite{Pe1} proved that:
\par
(a) if $n$ is odd and $a_1=a_3= \ldots =a_{n-2}=0$ then $a_0\le 1/e$,
\par
(b) if, besides the conditions listed in (a), $a_0=1/e$ actually holds then $a_n=2/e$.
\par
It should be noted that Brown's assumptions $a_i=0$ whenever $1\le i< (n+1)/2$ easily imply that also $b_i=0$ whenever $1\le i< (n+1)/2$, with the notation as in our Lemma~\ref{lem-exp}. It is also quite simple to check that, for $n$ odd, Peretz's assumptions $a_1=a_3= \ldots =a_{n-2}=0$ imply that $a_k=a_0 b_k$ for each odd $k\in \{1,2,\ldots,n\}$, hence we also have $b_1=b_3= \ldots =b_{n-2}=0$.
\par
Here we sketch a quick proof of both results in one stroke, without discussing  the case of equality. Recall that we are assuming that $a_0>0$ and $a_n>0$. This forces that $\mathrm{Re\,}b_0<0$ hence, by
periodicity of the exponential function we can impose the additional assumption $b_0<0$ without loss of generality. Such normalization, together with the recurrence relations \eqref{rec-exp}, yield
$$
 a_n = a_0 b_n = |a_0 b_n| = \left| \frac{b_n}{b_0} \right| |b_0| e^{-|b_0|}\,.
$$
The function $g/b_0$ belongs to the normalized class $P$ so Carath\'eodory's lemma applies: $|b_n/b_0|\le 2$; also,  as observed before, the function $x e^{-x}$ achieves its maximum $1/e$ at $x=1$ hence $|a_n|\leq 2/e$, as asserted by Krzy\.z. From here we can already deduce that for any normalized extremal function $b_0=-1$, $a_0=1/e$,  and we shall see later that this is enough to deduce that the only normalized extremal function is the conjectured one.
\par\smallskip
In what follows, the sets $I$ of indices as in the papers \cite{Br1}, \cite{Pe1} will be called inductive sets. This general approach will lead to further examples and a unified proof of the conjecture under other similar assumptions. We first introduce some notation and give a formal definition below. Fix $n\in \mathbb{N}$, $n\geq 2$. Given $K\subset\{1,2,3,...,n-1\}$, define
$$
 \C_K^n = \{c=(c_{1},c_{2},...,c_n) \in \C^n\,\colon\, c_i=0 \  \textrm{for \ all\ } i\in K\}\,.
$$
By an \textit{additive semigroup\/} or simply \textit{semigroup\/} we will mean a subset of $\N$ closed under addition. For $K\subset\{1,2,3,...,n-1\}$, denote by $G(K)$ the additive semigroup generated by $(K\cup \{n\})^{c} = \N \setminus (K\cup \{n\})$.
\begin{defn}\label{DefInductiveSets}
Let $n\in \mathbb{N}$, $n\geq 2$, and $a_{0}>0$. A subset $I$ of $\left\{ 1,2,3,...,n-1\right\}$ is said to be $n$-\emph{inductive} if $a_{n}=a_{0}b_{n}$ for \emph{all\/} $a\in \C_I^n$ and $b\in \C^n$ that satisfy the recursion formula \eqref{rec-exp}.
\par
A subset $J$ of $\left\{ 1,2,3,...,n-1\right\}$ is said to be \emph{exponentially} $n$-\emph{inductive} if $a_{n}=a_{0}b_{n}$ for \emph{all\/} $a\in \mathbb{C}^{n}$ and $b\in \mathbb{C}_J^{n}$ that satisfy the recursion formula \eqref{rec-exp}.
\end{defn}
We will sometimes suppress the integer $n$ and simply say $I$ is \textit{inductive\/} or $J$ is \textit{exponentially inductive\/} when the value of $n$ is understood. The following lemma helps us to identify inductive and exponentially inductive sets explicitly and easily.
\par
\begin{lem} \label{lem-ind}
Fix $n\in \mathbb{N}$, $n\geq 2$, $a_{0}>0$.
\par
(a) Let $I=\{i\,\colon\,1\le i\le n-1, a_i=0\}$. Then $I$ is $n$-inductive if and only if $n\notin G(I)$.
\par
(b) Let $J=\{j\,\colon\,1\le j\le n-1, b_j=0\}$. Then $J$ is exponentially $n$-inductive if and only if $n\notin G(J)$.
\end{lem}
\begin{proof}
(a) ($\Leftarrow$) Let $I\subset \left\{ 1,2,3,...,n-1\right\} $, and assume that $n\notin G(I)$. Let $Z = \left\{ 1,2,3,...,n-1\right\} \setminus G(I) \subset I$. Enumerate $Z$ as $Z =\left\{ z_{\ell }\right\}_{\ell =1}^{L}$ with $z_{1}<z_{2}<...<z_{L}<n$. We will show that $I$ is an inductive set by using induction on $\ell$ (in its finite version) to prove that
\begin{equation}\label{IfAIsZeroSoIsB}
 b_{z_{\ell }}=0\,, \qquad 1\leq \ell \leq L\,.
\end{equation}
To this end we will repeatedly use that, whenever $1\le j<n$, we have $j\in Z$ if and only if $j\notin G(I)$. For the inductive base case  $\ell=1$, note that, necessarily, $z_1 = 1$. Indeed, if not, then $1 \in G(I)$ and thus $n \in G(I)$, contradicting our assumption. Then we have from \eqref{rec-exp} that
\begin{equation*}
 0=a_{z_{1}}=a_1 = a_{0} b_{1}\,,
\end{equation*}
which implies that $b_{1}=0$ since $a_{0}$ is positive by assumption.
\par
Now we prove the inductive step that if  $b_{z_{k}}=0$ for all $k<\ell$ then also $b_{z_\ell}=0$. Indeed, assuming $b_{z_{k}}=0$ for all $k<\ell$, we have
\begin{equation}
 0=a_{z_{\ell }}=\sum_{j=1}^{z_{\ell }-1} \frac{j}{z_{\ell}} a_{z_{\ell }-j}b_{j}+ a_{0}b_{z_{\ell }}= a_{0}b_{z_{\ell }}\ ,  \label{identity}
\end{equation}%
since if $1\leq j\leq z_{\ell }-1$, then either $j\in Z$ and so $b_{j}=0$ by the inductive hypothesis, or $j\in G(I)$ and so $z_{\ell }-j\notin G(I)$ (since otherwise the semigroup property of $G(I)$ would give $z_{\ell}=j+z_{\ell}-j\in G(I)$, a contradiction), hence $z_{\ell}-j\in Z$. Thus $a_{z_{\ell }-j}=0$. In either case we have $a_{z_{\ell }-j}b_{j}=0$ and so \eqref{identity} holds. We now get $b_{z_{\ell}}=0$ since $a_{0}>0$. This completes the proof of \eqref{IfAIsZeroSoIsB}.
\par
Now we prove that $I$ is inductive from the same argument using that $n\notin G(I)$. Indeed,
\begin{equation*}
 a_{n}=\sum_{j=1}^{n-1} \frac{j}{n} a_{n-j}b_{j}+ a_{0}b_{n}= a_{0}b_{n}\ ,
\end{equation*}%
since if $j\in Z$ then $b_{j}=0$, while if $j\in G(I)$ then $n-j\notin G(I)$ and so $a_{n-j}=0$.
\par\smallskip
($\Rightarrow$) To prove the reverse implication, assume the contrary: $n \in G(I)$. We will show that then $G(I)$ is not $n$-inductive. In other words, we will show that for some $a\in \mathbb{C}_I^{n}$ and $b\in \mathbb{C}^{n}$ that satisfy the recursion formula \eqref{rec-exp} we will have $a_n\neq a_0 b_n$. To this end, start off with a pair of $n$-tuples $a\in \mathbb{C}_I^{n}$ and $b\in \mathbb{C}^{n}$ that satisfy \eqref{rec-exp}. If we have $a_n\neq a_0 b_n$ to begin with then there is nothing to prove, so assume that $a_n = a_0 b_n$.
\par
Since $n\in G(I)$ there exists $g_1$, $g_2$,\ldots,$g_s$ with $g_i \in (I\cup\{n\})^{c}$, $1\le i \le s$, and $g_1 + g_2 + \dots + g_s = n$. Note that $s\ge 2$ since none of the $g_i$ can be equal to $n$. We will now perturb some of the numbers $a_p$ while still requiring that the recurrence relations \eqref{rec-exp} hold, in such a way that the new point $a=(a_{1},a_{2},...a_{n}) \in \mathbb{C}_{I}^{n}$, but it is no longer true that $a_n = a_0 b_n$, or, equivalently, it is no longer true that $a_n - a_0 b_n = 0$. More precisely, denote the perturbed values $a_p$ by $\tilde{a}_p$, for $1\le p\le n$, and define
\begin{eqnarray} \label{PerturbationRelations}
 \tilde{a}_p & = & \left\{
 \begin{array}{ll}
 a_p + \varepsilon_p, &  \text{ if } p \in \{g_1,\ldots,g_s\}\,,  \nonumber \\
 a_p,  & \text{ if } p \in \{0,1,2,\ldots\,n\}\setminus\{g_1,\ldots,g_s\}\,.
 \end{array}
 \right.
\end{eqnarray}
(Since $a_0$ is fixed, we do not alter its value so we may formally understand that also $\tilde{a_0}=a_0$.) These values $\tilde{a}_p$ determine uniquely the corresponding new perturbed value $\tilde{b}_n$ according to Proposition~\ref{prop-polyn-b}:
\begin{eqnarray*}
 && \tilde{a}_n - a_0 \tilde{b}_n = \tilde{a}_n - a_0 Q_n (\tilde{a}_1,\tilde{a}_2,\ldots,\tilde{a}_n) =
\\
 && \tilde{a}_n - a_0 \displaystyle \sum_{\substack{1\le m\le n, \\  \sum_{j=1}^m i_j(n) = n}} (-1)^{m+1} c_{i_1(n),i_2(n),\ldots,i_m(n)}(a_0)\,\tilde{a}_{i_1(n)} \tilde{a}_{i_2(n)}\ldots\tilde{a}_{i_m(n)}\,.
\end{eqnarray*}
In view of our definition of $\tilde{a}_p$, the above value is a polynomial in the $s$ variables $\e_{g_1}$,\ldots,$\e_{g_s}$. Because of the assumption  that $g_1 + g_2 + \dots + g_s = n$, the polynomial above will contain a term with a non-zero coefficient, namely
\begin{eqnarray*}
 & & (-1)^{m+1}c_{g_1,g_2,\ldots,g_s}(a_0)\,\tilde{a}_{g_1} \tilde{a}_{g_2}\ldots \tilde{a}_{g_s}
\\
 &=& (-1)^{m+1} c_{g_1,g_2,\ldots,g_s}(a_0)\,(a_{g_1}+\e_{g_1}) (a_{g_2}+\e_{g_2}) \ldots (a_{g_s}+\e_{g_s})\,,
\end{eqnarray*}
which after an expansion will contain the term $\e_{g_1} \e_{g_2}\cdots\e_{g_s}$ that cannot possibly appear in any other summand.
\par
The Open Mapping Principle is well-known to hold for analytic functions from $\C^m$ to $\C$ and, in particular, for polynomials of several variables. Note that the above polynomial is a non-constant function because it contains at least one term whose corresponding coefficient does not vanish. Also, it takes on the value zero at the point $(\e_{g_1},\e_{g_2},\ldots,\e_{g_s})= (0,0,\ldots,0)$ by our assumption that $a_n = a_0 b_n$. Hence it follows that in a neighborhood of this point the polynomial takes on non-zero values. Thus, there is a perturbation that makes $\tilde{a}_n - a_0 \tilde{b}_n \neq 0$, and we are done.
\par\medskip
(b) ($\Leftarrow$) Completely analogous to the case (a), reversing the roles of $a_k$'s and $b_k$'s in the observations.
\par\smallskip
($\Rightarrow$) The converse is completely analogous to the case (a), using Proposition~\ref{prop-polyn-a} instead of Proposition~\ref{prop-polyn-b}, with the same idea involving perturbations and using the fact that all terms that should appear in the polynomial $P_n$ actually do appear because the relevant coefficients are non-zero.
\end{proof}
\par\smallskip
\begin{rk} \label{rk-main}
Note that the situation considered by Peretz \cite{Pe1} corresponds to the semigroup $G(I)=2\N=\{2,4,6,8,\ldots\}$. In Brown's result \cite{Br1}, $G(I)$ is the semigroup generated by the set $\{i\in\N\,\colon\, (n+1)/2\le i\le n-1\}$. In both cases, as observed before, $I=J$ hence $G(I)=G(J)$.
\end{rk}
\par
We would like to point out that different examples indeed exist. A general example of a pertinent semigroup is $G=\{k,2k,3k,\ldots\}$ for any fixed $k\ge 2$, $k\in\N$, such that $n$ is not a multiple of $k$. An even more general family of examples is obtained by choosing $1<a\le b$ and letting $G_{a,b}$ be the semigroup generated by $\left\langle a,b\right\rangle \equiv \{k\in\N\,\colon\,a\le k\le b\}$. It is not difficult to see that
$$
 G_{a,b}=\bigcup_{\ell =1}^{\infty }\left\langle \ell a,\ell
b\right\rangle \ .
$$
Whenever it is possible to choose $k>1$ so that
$kb+2=\left( k+1\right) a$, we can consider the value $n=kb+1$ so that
$$
 G_{a,b}\cap \left\langle 1,n\right\rangle =\left\langle a,b\right\rangle
 \overset{\cdot }{\cup }\left\langle 2a,2b\right\rangle \overset{\cdot }{\cup
 }\left\langle 3a,3b\right\rangle \overset{\cdot }{\cup }...\overset{\cdot }{\cup }\left\langle ka,kb\right\rangle .
$$
Now we compute the density of $I$ in the set $\{1,2,\ldots,n-1\}$ to be
\begin{eqnarray*}
 \frac{n-1-\#\left( G_{a,b}\cap \left\langle 1,n-1\right\rangle \right) }{n-1} &=& \frac{n-1-\sum_{\ell =1}^{k}\left[ \ell \left( b-a\right) +1\right] }{n-1}
\\
 &=& \frac{n-1-\left[ \frac{k\left( k+1\right) }{2}\left( b-a\right) +k\right]}{n-1}
\\
 &=&\frac{n-1-\frac{1}{2}\left[ \left( k+1\right) \left( a-2\right) +2k\right]}{n-1}
\\
 &=&\frac{kb-\frac{1}{2} kb}{n-1} = \frac{1}{2}\,.
\end{eqnarray*}%
Note that density $\frac{1}{2}$ is the smallest density needed to deduce  $a_n=a_0 b_n$ following the above methods, since the fact that $j + (n-j) = n \notin G(I)$ implies that either $j$ or $(n-j)$ is in $I$, hence $I$ must have density at least $\frac{1}{2}$ within the set $\{1,2,\ldots,n-1\}$.
\par\medskip

\section{The main result and its proof}
\par
At this point it is convenient to summarize some of the findings on extremal functions for the normalized Krzy\.z problem \eqref{eqn-extr-ref}. Recall that this normalization requires that $a_0>0$, hence $\mathrm{Re\,}b_0<0$. As observed before, due to periodicity of the exponential function, without loss of generality we may actually assume that $b_0\in\R$ and $b_0<0$. We know from our earlier discussions that any such function $f$ fulfills the following conditions:
\begin{itemize}
\item[(i)]
$a_n>0$ (in fact, $a_n=M_n\ge 2/e$).
\item[(ii)]
$2\le \frac{a_n}{a_0} \le 1 + \frac1{B_0}$, where $B_0$ is the constant term in the Blaschke factor of the factorization given in \eqref{factor} normalized so that $0<B_0\le 1$.
\item[(iii)]
The polynomial $P(z) = a_n + 2 a_{n-1} z + \ldots 2 a_1 z^{n-1} + 2 a_0 z^n$ from Proposition~\ref{prop-var} has non-negative real part on the closed unit disk $\overline{\D}$ and strictly positive real part on $\D$.
\item[(iv)]
$N\le n$ in formula \eqref{str-extr-fcn} and also Re\,$b_n\le 2 |b_0|$ (the function $g/b_0$ has positive real part in $\D$ and value one at the origin; by Carath\'eodory's lemma \cite[p.~41]{Du2}, its Taylor coefficients are bounded by two).
\item[(v)]
The zeros $\la_j$, $1\le j\le n$, of the polynomial $P$ satisfy $|\la_k|\ge 1$ for $1\le k\le n$ and \eqref{prod-roots}.
\end{itemize}
\par
Our aim is to show that, essentially, if equality holds in any one of the above inequalities, then the conjecture is true. We are now ready to state and prove our main result.
\par
\begin{thm}
 \label{thm-equiv}
Let $n\ge 2$ and consider an arbitrary but fixed extremal function $f$ for the Krzy\.z problem \eqref{eqn-extr-ref}. Writing $f(z)= \sum_{j=0}^\infty a_j z^j$, $g(z)= \sum_{j=0}^\infty b_j z^j$, and $f=e^g$ as before, we know that $a_n$, $a_0>0$ and may also assume without loss of generality that $b_0<0$. Consider the quantity $B_0$, the polynomial $P$ and its zeros as described above.
\smallskip\par\noindent
$\mathrm{(I)}$ The following statements are equivalent:
\begin{itemize}
\item[(a)]
$a_n=2 a_0$;
\item[(b)]
$a_k=0$ when $1\le k<n$ (equivalently by \eqref{rec-exp}, $b_k=0$ when $1\le k<n$);
\item[(c)]
$f(z)=e^{(z^n-1)/(z^n+1)}$ (and, in particular, $M_n=2/e$);
\item[(d)]
the set $I$ consisting of all indices $i\in\{1,2,\ldots,n-1\}$ for which $a_i=0$ is $n$-inductive;
\item[(e)]
the set $J$ consisting of all indices $j\in\{1,2,\ldots,n-1\}$ for which $b_j=0$ is exponentially $n$-inductive;
\item[(f)]
$g(z)=(z^k H(z)-1)/(z^k H(z)+1)$ for some analytic function $H$ in $\D$ such that $|H(z)|\le 1$ for all $z\in\D$ and $k\in\N$, $k\ge n/2$;
\item[(g)]
the zeros of the polynomial $P$ satisfy $(-1)^n\prod_{k=1}^n \la_k=1$;
\item[(h)]
the zeros of $P$ all lie on the unit circle;
\item[(i)]
the zeros of $P$ are actually the $n$-th roots of $-1$;
\item[(j)]
$a_n b_0+a_0 \mathrm{Re\,}\{b_n\}=0$;
\item[(k)]
$\mathrm{Re\,}\{a_1 b_{n-1} +a_2 b_{n-2} + \ldots + a_{n-1} b_1\}=0$;
\item[(l)]
$a_n=a_0\,\mathrm{Re}\,b_n$.
\item[(m)]
Re\,$b_n=2 |b_0|$, $N=n$, and $r_1=r_2=\ldots =r_n$;
 \item[(n)]
$B_0=1$;
 \item[(o)]
$W_n f$ does not vanish in $\D$;
 \item[(p)]
$W_n f\equiv f$;
 \item[(q)]
$g=W_n g$.
\end{itemize}
\smallskip\par\noindent
$\mathrm{(II)}$ In addition to the above, the following is true: there is a unique extremal function for \eqref{eqn-extr-ref} if and only if every  extremal function for \eqref{eqn-extr-ref} satisfies any one of the conditions (a)--(q) from part (I), and therefore all of them.
\end{thm}
\par\medskip
\textbf{Some remarks}.
\begin{itemize}
\item
The implication (d)\,$\Rightarrow$\,(c), which is a consequence of our theorem, is thus a generalization of Brown's result \cite{Br1} for $p=\infty$ and an improvement of the result of Peretz \cite{Pe1} mentioned earlier since it yields directly that if $f$ is extremal, $n$ is odd, and $a_1=a_3=\ldots=a_{n-2}=0$ then $a_0=1/e$ and $a_n=2/e$.
\item
According to the above findings, proving the conjecture amounts to  showing that $N=n$ and the following sets of numbers coincide:
\par
- $\{\a_1,\ldots,\a_N\}$, the rotation coefficients in the point masses in extremal functions as in \eqref{str-extr-fcn},
\par
- $\{\om_1,\ldots,\om_n\}$, the $n$-th roots of $-1$,
\par
- $\{\la_1,\ldots,\la_n\}$, the roots of the polynomial $P$ associated with the extremal function $f$,
\par
- the zeros of Re\,$P=|Q|^2$ on $\T$.
\item
Alternatively, it suffices to show the uniqueness of the extremal function for \eqref{eqn-extr-ref}. According to D.~Khavinson, this so far unpublished fact was already known earlier to various experts, for example, to Stephen D. Fisher.
\item
It should be noted that neither of the conditions $n\notin G(I)$, $n\notin G(J)$, apparently equivalent to those from the above list, is included in the theorem. The reason is that in both statements (a) and (b) in Lemma~\ref{lem-ind} one implication was proved without checking whether the perturbed coefficients actually correspond to an admissible function (that is, to one for which Re\,$g<0$ in $\D$) and it is unclear whether that implication also holds for the more restricted class of coefficients of admissible function as opposed to the class of coefficients considered in Lemma~\ref{lem-ind}. However, Lemma~\ref{lem-ind} provides sufficient conditions for $I$ being $n$-inductive or $J$ being exponentially $n$-inductive that are easy to check.
\end{itemize}
\par
\begin{proof} \underline{\textbf{Part (I)}}. The scheme of the proof is as follows. We first show that (a) $\Rightarrow$ (b) $\Rightarrow$ (c) $\Rightarrow$ (i) $\Rightarrow$ (h) $\Rightarrow$ (g) $\Rightarrow$ (a). We will then see that (c) $\Rightarrow$ (m) $\Rightarrow$ (b).
\par
Next, we will see that (c) $\Rightarrow$ (d) $\Rightarrow$ (l) and also (c) $\Rightarrow$ (e) \,$\Rightarrow$\,(l) $\Rightarrow$ (j)\,$\Rightarrow$\, (k)\,$\Rightarrow$\,(a) and this will close some further loops.
\par
Then we will verify that (c) $\Leftrightarrow$ (f).
\par
Finally, we will show that (c) $\Rightarrow$ (p) $\Rightarrow$ (o) $\Rightarrow$ (n) $\Rightarrow$ (a) and will also check that (c) $\Rightarrow$ (q) $\Rightarrow$ (b), which will complete the equivalence of the 17 conditions (a)--(q).
\par
We note that all arguments are quite brief (due to the work done previously) with the exception of the first two, which together account for about two pages.
\par\medskip
\fbox{(a) $\Rightarrow$ (b):} Let $f$ be an extremal function. By Proposition~\ref{prop-var}, we know that the coefficients of $f$ satisfy the following condition:
\begin{equation}
 \mathrm{Re\,} \{a_n + 2 a_{n-1} \la + \ldots + 2 a_1 \la^{n-1} + 2 a_0 \la^n\} \ge 0\,, \quad \textrm{whenever } |\la|\le 1\,.
 \label{coeff-cond}
\end{equation}
Let us write
$$
 \om_k = e^{(2k+1) \pi i / n}\,, \quad k= 0, 1,\ldots, n-1\,,
$$
for the $n$-th roots of $-1$. Our assumption that $a_n=2
a_0$ shows that whenever $\la=\om_k$, $k= 0, 1,\ldots, n-1$, we have
$\mathrm{Re\,} \{a_n + 2 \om_k^n a_0\} = 0$. Hence from
\eqref{coeff-cond} we conclude that for each of these values
\begin{equation}
 \mathrm{Re\,} \{a_{n-1} \om_k + \ldots + a_1 \om_k^{n-1}\} \ge 0\,, \quad k= 0, 1,\ldots, n-1\,.
 \label{coeff-simple}
\end{equation}
By basic algebra, for any fixed $j$ with $1\le j\le n-1$ we have
$$
 \sum_{k=0}^{n-1} \om_k^j = e^{\pi i j/ n} \sum_{k=0}^{n-1} e^{2 k j \pi i / n}	
 = e^{\pi i j/ n} \frac{1 - e^{2 j \pi i}}{1 - e^{2 j \pi i / n}} = 0\,.
$$
Thus, summing up all terms that appear on the left in \eqref{coeff-simple} over $k=0,1,\ldots,n-1$, we get
$$
 \sum_{k=0}^{n-1} \mathrm{Re\,} \{a_{n-1} \om_k + a_{n-2} \om_k^2 + \ldots + a_1 \om_k^{n-1}\} = \mathrm{Re\,} \sum_{j=1}^{n-1} \left\{a_{n-j} \sum_{k=0}^{n-1} \om_k^j \right\} = 0\,.
$$
Since every summand on the left-hand side in the above formula is non-negative by \eqref{coeff-simple}, all of them must be zero:
\begin{equation}
 \mathrm{Re\,} \{a_{n-1} \om_k + a_{n-2} \om_k^2 + \ldots + a_1 \om_k^{n-1}\} = 0\,, \quad
 k = 0, 1,\ldots, n-1\,,
 \label{coeff-zero}
\end{equation}
hence also
\begin{equation}
 \mathrm{Re\,} \{a_n + 2 a_{n-1} \om_k + \ldots + 2 a_1 \om_k^{n-1} + 2 a_0 \om_k^n\} = 0\,, \quad k= 0, 1,\ldots, n-1\,,
 \label{all-zero}
\end{equation}
in view of our choice of $\om_k$ and the assumption that $a_n=2 a_0$.
\par
As remarked before, the function
$$
 T(t) = \mathrm{Re\,} \{a_n + 2 a_{n-1} \la + \ldots + 2 a_1 \la^{n-1} + 2 a_0 \la^n\}
$$
is a trigonometric polynomial of degree $n$ of the variable $t\in
[0,2\pi]$, where $\la=e^{i t}$. Since $T(t)\ge 0$ on the circle,
Fej\'er's Lemma tells us that for some coefficients $\g_0$,$\g_1$,\ldots,$\g_n$ we have
$$
 T(t)=|(\g_0+\g_1 \la +\ldots \g_n \la^n)^2|\,, \quad \la=e^{i t}\,.
$$
The complex polynomial $Q(z)=(\g_0+\g_1 z +\ldots \g_n z^n)^2$ has $2n$ zeros counting the multiplicities, each zero being obviously of order
at least two. But we know from \eqref{all-zero} that this polynomial
has at least $n$ distinct zeros $\om_k$, $k= 0, 1,\ldots, n-1$, which
are roots of $-1$, so each one of these zeros must be double and hence
$Q$ cannot have any other zeros. Thus, the polynomial factorizes as
$$
 Q(z) = (\g_0+\g_1 z +\ldots \g_n z^n)^2 = C \prod_{k=0}^{n-1} (z -
 \om_k)^2 = C (z^n+1)^2 \,.
$$
Hence
$$
 \mathrm{Re\,} \{a_n + 2 a_{n-1} \la + \ldots + 2 a_1 \la^{n-1} + 2 a_0
 \la^n\} = |C (\la^n+1)^2| = 2 |C|\,\mathrm{Re\,} \{\la^n+1\}
$$
for all $\la$ on the unit circle. As was observed earlier, two
polynomials whose real parts are equal on the unit circle must coincide
everywhere, except for an imaginary constant:
$$
 a_n + 2 a_{n-1} z + \ldots + 2 a_1 z^{n-1} + 2 a_0 z^n = 2 |C|
 (z^n+1) + ic\,, \qquad z\in\C\,, \quad c\in\R\,,
$$
but since we know that actually $a_n>0$, we finally have
$$
 c=0\,, \quad a_{n}=2 a_0 = 2 |C|\,, \quad a_1=a_2=\ldots=a_{n-1}=0\,, $$
which yields (b).
\par\smallskip
\fbox{(b) $\Rightarrow$ (c):} In view of the inequalities $a_0$,  $a_n>0$ and $b_0<0$, the recurrence relations \eqref{rec-exp} yield
$$
 a_n = a_0 b_n = |a_0 b_n| = \left| \frac{b_n}{b_0} \right| |b_0| e^{-|b_0|}\,.
$$
The function $g/b_0$ belongs to the normalized class $P$ so Carath\'eodory's lemma applies: $|b_n/b_0|\le 2$. The function $x e^{-x}$ achieves its maximum $1/e$ at $x=1$. Thus, we obtained the desired inequality $|a_n|\leq 2/e$.
\par
It is only left to discuss the case of equality. When this happens, we must have $|b_0|=1$, hence $b_0=-1$ and $a_0=1/e$. Also, in order for equality to hold in Carath\'eodory's lemma: $|b_n/b_0|=2$, the measure in the Herglotz representation of the exponent $g$ must supported on a set where $e^{-i n t}$ has constant signum. Thus, each of the numbers $\a_j$, $1\le j\le N$, is one of the $n$-th roots of some $\z$ such that $|\z|=1$. By the geometric series expansion of the terms in the exponent and by comparing coefficients we get
$$
 b_0=-\sum_{j=1}^N r_j\,, \quad b_1=-2\sum_{j=1}^N r_j \a_j\,, \ldots, \quad b_{n-1}=-2\sum_{j=1}^N r_j \a_j^{n-1}\,.
$$
We already know that $b_0=-1$ and $b_1=\ldots=b_{n-1}=0$, hence
$$
 \sum_{j=1}^N r_j=1\,, \quad \sum_{j=1}^N r_j \a_j=0\,, \ldots, \quad  \sum_{j=1}^N r_j \a_j^{n-1}=0\,.
$$
We may enlarge the set $\{\a_j\,\colon\,1\le j\le N\}$ so as to include all the $n$-th roots of $\z$ (if $N<n$) and then may reorder it so as to obtain the complete system
$$
 \sum_{j=1}^n r_j=1\,, \quad \sum_{j=1}^n r_j \a_j=0\,, \ldots, \quad  \sum_{j=1}^n r_j \a_j^{n-1}=0
$$
of $n$ linear equations in $n$ unknowns $r_1$,\ldots,$r_n$, with $r_j=0$ for $N+1\le j\le n$. The determinant of this system is the Vandermonde determinant
$$
 V_n = \begin{vmatrix} 1 & 1  & \cdots & 1
\\
  \a_1 & \a_2 & \cdots & \a_n
\\
  \a_1^2 & \a_2^2 & \cdots & \a_n^2
\\
 \vdots & \vdots & \ddots & \vdots
\\
 \a_1^{n-1} & a_2^{n-1} & \cdots & \a_n^{n-1}
\end{vmatrix}
 = \prod_{1 \mathop \le i \mathop < j \mathop \le n} \left({\a_j - \a_i}\right) \neq 0\,,
$$
hence the system has a unique solution. In view of the properties of the sums of powers of the $n$-th roots of $\z$, the system is obviously satisfied when
$$
 (r_1,\ldots,r_n) = \(\frac1{n},\ldots,\frac1{n}\)
$$
and thus it follows that $N=n$. Let $\om = e^{(2\pi i)/n}$ be the primitive $n$-th root of $1$. Then $\{\a_1,\ldots,\a_n\} = \{\om^k\a_1,\ldots,\om^k \a_n\}$ for any fixed $k$ with $0\le k<n$ and therefore the function
$$
 g(z) = \frac1{n} \sum_{j=1}^n \frac{\a_j z+1}{\a_j z-1}
$$
has the property that $g(\om^k z)=g(z)$ for all $z$ in $\D$ and $0\le k<n$. Hence
$$
 W_n g(z) =\frac{ \sum_{k=0}^{n-1} g(z)}{n} = g(z)\,,
$$
so $g$ is a function of $z^n$ so $g$ is a function of $z^n$. Given our normalizations and that $d(z) = \prod_{j=1}^{n} \left( \alpha_j z -1 \right)$ also satisfies $W_n d (z) = d(z)$, which can be proved as we did for $g$, it now readily follows that
$$
  g(z) = \frac{z^n-1}{z^n+1}
$$
and this proves (c).
\par\smallskip
\fbox{(c) $\Rightarrow$ (i)} is easy since for the function given
by (c) we have
$$
 f(z) = e^{(z^n-1)/(z^n+1)} = \frac1{e}+\frac2{e} z^n +\ldots\,,
$$
hence $P(z)=\frac2{e} (z^n+1)$. This means that the numbers $\la_k$ are
precisely the $n$-th roots of $-1$.
\par\smallskip
\fbox{(i) $\Rightarrow$ (h)} is completely trivial.
\par\smallskip
\fbox{(h) $\Rightarrow$ (g)} clearly follows from \eqref{prod-roots}.
\par\smallskip
\fbox{(g) $\Rightarrow$ (a):} The assumption (g) means that
$(-1)^n \prod_{k=1}^n \la_k=1$. In view of \eqref{a_n-a_0}, this means
that $a_n=2 a_0$.
\par\medskip
\fbox{(c)\,$\Rightarrow$ (m):} Denoting again by $\om_j$ the $n$-th roots of $-1$, this follows from the identity
$$
 e^{\frac{z^n-1}{z^n+1}} = e^{\frac1{n} \sum_{j=0}^n \frac{\om_j z+1}{\om_j z-1}}
$$
which can be proved as earlier.
\par\medskip
\fbox{(m)\,$\Rightarrow$ (b):} Starting from formula \eqref{str-extr-fcn}, the geometric series expansion shows that
$$
 b_0=-\sum_{j=1}^n r_j\,, b_1=-2\sum_{j=1}^n r_j \a_j\,, b_2=-2\sum_{j=1}^n r_j \a_j^2\,, \ldots, b_n=-2\sum_{j=1}^n r_j \a_j^n\,.
$$
Thus,
$$
 \mathrm{Re\,} b_n \le |b_n| = 2 \left| \sum_{j=1}^n r_j \a_j^n \right| \le 2 \sum_{j=1}^n r_j = 2 |b_0|\,.
$$
By our assumption, equality must hold throughout but this is only possible if the signum of $\a_j^n$ is independent of $j$ and $b_n>0$, meaning that $\a_j^n=-1$ for all $j$. But then, taking into account the assumption that $r_1=r_2=\ldots =r_n$, the basic algebra of complex numbers and the above formulas show that
$$
 b_1=b_2=\ldots=b_{n-1}=0\,,
$$
the desired conclusion (b) follows by Lemma~\ref{lem-exp}. (Note that we could have also used the condition for equality in Carath\'eodory's lemma.)
\par\smallskip
\fbox{(c)\,$\Rightarrow$ (d):} It is clear that from (c) we get $b_1= \ldots =b_{n-1}=0$, hence $a_1=\ldots=a_{n-1}=0$ by \eqref{rec-exp}; that is, $I=\{1,2,\ldots,n-1\}$ in this case. Since \eqref{rec-exp} readily yields $a_n=a_0 b_n$, the set $I$ is clearly $n$-inductive.
\par\smallskip
\fbox{(d)\,$\Rightarrow$ (l):} If $I$ is $n$-inductive then $a_n=a_0 b_n$ and (l) follows trivially since $a_n$, $a_0>0$ by assumption.
\par\smallskip
\fbox{(c)\,$\Rightarrow$ (e):} From (c), $b_1= \ldots =b_{n-1}=0$, hence $J=\{1,2,\ldots,n-1\}$ in this case. Then \eqref{rec-exp} implies $a_n=a_0 b_n$, hence $J$ is exponentially $n$-inductive.
\par\smallskip
\fbox{(e)\,$\Rightarrow$ (l):} If $J$ is exponentially $n$-inductive then $a_n=a_0 b_n$ and (l) again follows trivially in view of $a_n$, $a_0>0$.
\par\smallskip
\fbox{(l)\,$\Rightarrow$ (j):} Assume that $a_n=a_0$\,Re\,$b_n$ holds. Recalling that $a_0=e^{b_0}=e^{-|b_0|}$, we have $a_n= |b_0| e^{-|b_0|} \dfrac{\mathrm{Re\,}b_n}{|b_0|}$. By Carath\'eodory's lemma for analytic functions with positive real part and value one at the origin (applied to the function $g/b_0$), we have
$$
 \frac{\mathrm{Re\,}b_n}{|b_0|} \le \left|\frac{b_n}{b_0}\right| \le 2\,.
$$
Since the maximum value of $x e^{-x}$ is $1/e$ and is attained at $x=1$, we get $a_n\le 2/e$ and hence $a_n=2/e$, so equality holds above. Thus, we must also have  $b_0=-1$ and (j) follows readily.
\par\smallskip
\fbox{(j)\,$\Rightarrow$ (k):} Immediate from \eqref{cond-coeff}.
\par\smallskip
\fbox{(k)\,$\Rightarrow$ (a):} Apply \eqref{cond-coeff} again to get Re\,$\{a_0 b_n +a_n b_0\} = 0$. Thus, recalling that $a_0$, $a_n>0$ and $b_0<0$, it follows that
\[
 a_n = a_0\,\mathrm{Re\,}\left\{\frac{b_n}{-b_0}\right\} \le a_0 \left| \frac{b_n}{b_0} \right| \le 2 a_0
\]
by Carath\'eodory's lemma applied to $g/b_0$. Recalling that $a_n\ge 2 a_0$, (a) follows.
\par\smallskip
\fbox{(c)\,$\Rightarrow$ (f):} Obviously true with $k=n$ and $H\equiv 1$.
\par\smallskip
\fbox{(f)\,$\Rightarrow$ (c):} If $f$ is as in the assumptions of condition (f) and we write $f=e^g$ as before, then
\begin{eqnarray*}
 g(z) &=& - 1 + 2 z^k H(z) - 2 z^{2 k} H(z)^2 + \ldots
\\
 &=& -1 +b_k z^k + b_{k+1} z^{k+1} + \ldots + b_n z^n + \ldots
\end{eqnarray*}
In other words, $b_1=b_2=\ldots=b_{k-1}=0$. If $n$ is odd then obviously $k\ge (n+1)/2$, hence it follows that $b_i=0$ whenever $1\le i<(n+1)/2$ and then it follows from \eqref{rec-exp} that also  $a_i=0$ whenever $1\le i<(n+1)/2$. Then Brown's result (whose hypotheses are a special case of our condition (d) because of Lemma~\ref{lem-ind}) implies the conjecture and hence our condition (c).
\par
The more interesting case is when $n$ is even and $k=n/2$ when the hypotheses of our condition (d) are not automatically fulfilled and hence there is something to prove. (The case $k>n/2$ is, of course, easier and follows from (d) and Lemma~\ref{lem-ind} because again it is part of Brown's result mentioned earlier.) Writing $n=2k$ and taking into account that then $b_1=b_2=\ldots=b_{k-1}= 0= a_1=a_2=\ldots=a_{k-1}$, we again have by \eqref{rec-exp} that $a_k=b_k a_0$ and, in view of $a_0=e^{-1}$, this leads to
$$
 a_{2k} = \sum_{j=1}^{2k} \frac{j}{2k} a_{2k-j} b_j = \frac12 a_k b_k + a_0 b_{2k} = \(\frac12 b_k^2 + b_{2k} \) a_0 = \frac1{e} \(\frac12 b_k^2 + b_{2k} \)\,.
$$
The inequality $a_n=a_{2k}\le 2/e$ now follows by Lemma~\ref{lem-livingst}. Since we already know that for our extremal function $a_n\ge 2/e$, it follows that $a_n=2/e=2 a_0$, which is (a), and we already know that this implies (c).
\par\smallskip
\fbox{(c)\,$\Rightarrow$ (p):} Follows because the (only) extremal function $f$ in (c) is an analytic function of $z^n$.
\par\smallskip
\fbox{(p)\,$\Rightarrow$ (o):} Obvious.
\par\smallskip
\fbox{(o)\,$\Rightarrow$ (n):} If $W_n f$ does not vanish in the disk, the Blaschke factor $B$ of $H$ is constant: $B_0=1$, which is (n).
\par\smallskip
\fbox{(n)\,$\Rightarrow$ (a):} Follows from Proposition~\ref{prop-var} and Proposition~\ref{prop-wien} which combined yield $2\le \dfrac{a_n}{a_0}\le 2$.
\par\smallskip
\fbox{(c)\,$\Rightarrow$ (q):} Clear from the fact that the exponent is an analytic function of $z^n$.
\par\smallskip
\fbox{(q)\,$\Rightarrow$ (b):} Under the assumption (q), we have (for example, by Lemma~\ref{lem-exp})
$$
 f(z) = e^{b_0 + b_n z^n + \ldots} = e^{b_0} + b_n e^{b_0} z^n + \ldots \,,
$$
which readily implies (b).
\par\bigskip
\underline{\textbf{Part (II)}.}
\par
Suppose first that the extremal function for \eqref{eqn-extr-ref} is unique and denote it by $f$, keeping the notation as in part (I). Let $\om = e^{(2\pi i)/n}$ and define $g_j(z)=f(\om^j z)$, for $j=1$, \ldots, $n$. It is clear from the Taylor series of $f$ that
the constant term  of $g_j$ is $a_0$ and its $n$-th coefficient is $a_n \om^{j n}=a_n$. In view of the uniqueness of our extremal function, it follows that $g_j=f$ for all values of $j$. Hence Wiener's trick yields
$$
 W_n f (z)= \frac1{n} \sum_{j=0}^{n-1} f(\om^j z) = \frac1{n} \sum_{j=0}^{n-1} g_j (z) = f(z)
$$
so $W_n f\in \cb_\ast$ and the $n$-th coefficient of $W_n f$ is $a_n$ but at the same time also equals $H^\prime (0)$ for the function $H$ defined by \eqref{fcn-h} as before, hence by Lemma~\ref{lem-n=1}, it is bounded by $2/e$. Moreover, under the normalization \eqref{eqn-extr-ref} imposed, equality holds if and only if $H(z)= e^{(z-1)/(z+1)}$ by the remark following Lemma~\ref{lem-n=1}. That is, if and only if
\[
 W_n f(z)= e^{\frac{z^n-1}{z^n+1}}\,.
\]
In view of the equality $f(z)=W_n f(z)$, condition (c) from part (I) follows. Recalling that all conditions (a)--(q) are equivalent, $f$ must satisfy all other conditions as well.
\par
Conversely, if any of the conditions from part (I) holds, then also (c) holds. But condition (c) gives the uniqueness of the extremal function. This ends the proof.
\end{proof}					
\par\medskip
\textbf{Closing remarks}.
\begin{itemize}
\item
At the present time, we are not able to deduce the desired equality $a_n= 2 a_0$ for an extremal function.
\item
Also, it is not clear to us how one can show that $a_0=1/e$ holds for extremal functions nor whether this equality implies the other conditions without further assumptions on extremal functions.
\item
At this point we are not able to show the uniqueness of extremal functions either.
\end{itemize}
\par
In summary, the Krzy\.z conjecture remains open. However, it is our hope that the statements proved here may point out in some new directions for further research on the problem.


\end{document}